\documentclass[12pt]{article}
\usepackage[T1]{fontenc}
\usepackage{lmodern,amsmath,amsthm,amsfonts,amssymb,graphicx,float,microtype,thmtools,underscore,mathtools,xurl,multirow, lineno}

\usepackage{algorithm} 
\usepackage{algpseudocode} 

\usepackage[usenames,dvipsnames,svgnames,table]{xcolor}
\usepackage[shortlabels]{enumitem}
\setlist[itemize]{topsep=0ex,itemsep=0ex,parsep=0ex}
\setlist[enumerate]{topsep=0ex,itemsep=0ex,parsep=0ex}
\usepackage{pgf,tikz,ifthen,arrayjobx}
\usetikzlibrary{arrows}
\usetikzlibrary{calc}
\usepackage[unicode=true]{hyperref}
\hypersetup{ 
colorlinks,
breaklinks=true,
linkcolor={blue!60!black},
citecolor={black},
urlcolor={blue!60!black},
pdftitle={Notes on Graph Covers}}
\usepackage[capitalise, compress, nameinlink, noabbrev]{cleveref}
\crefname{lem}{Lemma}{Lemmas}
\crefname{thm}{Theorem}{Theorems}
\crefname{prop}{Proposition}{Propositions}
\crefname{cor}{Corollary}{Corollaries}

\usepackage[longnamesfirst,numbers,sort&compress]{natbib}
\makeatletter
\def\NAT@spacechar{~}
\makeatother
\usepackage[margin=30mm]{geometry}
\renewcommand{\baselinestretch}{1.15}
\allowdisplaybreaks

\renewcommand{\epsilon}{\varepsilon}
\renewcommand{\emptyset}{\varnothing}

\renewcommand{\geq}{\geqslant}
\renewcommand{\leq}{\leqslant}

\renewcommand{\thefootnote}{\fnsymbol{footnote}}
\theoremstyle{plain}
\newtheorem{thm}{Theorem}

\newtheorem{cor}[thm]{Corollary}
\newtheorem{prop}[thm]{Proposition}

\newtheorem*{lem*}{Lemma}
\theoremstyle{definition}

\newtheorem*{conj*}{Conjecture}
\newtheorem{remark}{Remark}
\newtheorem{observation}{Observation}
\newtheorem{example}[thm]{Example}
\theoremstyle{problem}

\begin{document}
\title{\bf\boldmath\fontsize{18pt}{18pt}\selectfont Domination Parameters of Graph Covers}

\author{%
Dickson~Y.~B. Annor\,\footnotemark[5] 
\qquad
}

\date{}

\maketitle

\begin{abstract}

A graph $G$ is a \emph{cover} of a graph $F$ if there exists an onto mapping $\pi : V(G) \to V(F)$, called a (\emph{covering}) \emph{projection}, such that $\pi$ maps the neighbours of any vertex $v$ in $G$ bijectively onto the neighbours of $\pi(v)$ in $F$. This paper is the first attempt to study the connection between domination parameters and graph covers. We focus on the domination number, the total domination number, and the connected domination number. We prove upper and lower bounds for the domination parameters of $G$.
\end{abstract}



\footnotetext[5]{Department of Mathematical and Physical Sciences, La Trobe University, Bendigo, Australia (\texttt{d.annor@latrobe.edu.au}).
Research of Annor supported by a La Trobe Graduate Research Scholarship. 
}


\renewcommand{\thefootnote}{\arabic{footnote}}

\section{Introduction}

In this paper, we deal with finite undirected and simple graphs. 
For a graph $F$, let $V(F)$ and $E(F)$ respectively denote the vertex set and the edge set of $F$.

A graph $G$ is a \emph{cover} of a graph $F$ if there exists an onto mapping $\pi : V(G) \to V(F)$, called a (\emph{covering}) \emph{projection}, such that $\pi$ maps the neighbours of any vertex $v$ in $G$ bijectively onto the neighbours of $\pi(v)$ in $F$. Note that every graph is a cover of itself via the identity projection. 
If $F$ is connected, then $|\pi^{-1}(v)| = k$ is the same for all $v \in V(F)$ and $\pi$ is called a \emph{$k$-fold} cover. We often call a $2$-fold cover a \emph{double cover}. See Figure~\ref{fig:petersencover} for an example of a cover. Here, we can define the vertex projection $\pi$ by labelling every vertex $v$ of $G$ with their image $\pi(v)$. We often add a dash or subscript to distinguish between two or more vertices of $G$ projected to the same vertex of $F$. Let graph $G$ be a cover of graph $F$. For any subgraph $F'$ of $F$, we call the graph $G' = \pi^{-1}(F')$ the \emph{lift of} $F'$ \emph{into G}.

The following proposition illustrates some properties of covers.
\begin{prop}\label{prop:pertiesgc}
   Let graph $G$ be a cover of graph $F$ via the projection $\pi$. 
\begin{enumerate}[label=(\roman*),ref=\roman*]
    \item\label{it:dgpreserve}  It holds that $\deg_{G}(v) = \deg_{F}(\pi(v))$ for each vertex $v \in V(G)$, where $\deg_G(v)$ is the degree of a vertex $v$ in $G$.
    \item\label{it:compt} The lift of a tree $T$ into $G$ consists of a collection of disjoint trees isomorphic to $T$. Hence, a cover of a tree is a forest.
    \item\label{it:liftcycles} The lift of  a cycle $C_n$ into $G$ consists of a collection of disjoint cycles whose lengths are multiples of $n$.
\end{enumerate}
\end{prop}\label{pro:coverpp}

\begin{figure}
    \centering
    \begin{tikzpicture}[scale=0.75]
\begin{scope}[shift ={(9.5, 0)}]
  \draw[ thick] (0,0) -- (3,2.5);   
\draw[ thick] (3,2.5) -- (6,0);
\draw[ thick] (1,-3) -- (5,-3);
\draw[ thick] (0,0) -- (1,-3);
\draw[ thick] (5,-3) -- (6,0);
\draw[ thick] (0,0) -- (1.5,0);
\draw[ thick] (3,2.5) -- (3,1);
\draw[ thick] (6,0) -- (4.5,-0.5);
\draw[ thick] (5,-3) -- (4,-2);
\draw[ thick] (1,-3) -- (2,-2);
\draw[ thick] (1.5,0) -- (4.5,-0.5);
\draw[ thick] (4.5,-0.5) -- (2,-2);
\draw[ thick] (2,-2) -- (3,1);
\draw[ thick] (3,1) -- (4,-2);
\draw[ thick] (4,-2) -- (1.5,0);
 \filldraw[black] (0,0) circle (2.5pt) node[anchor=east]{$1$};
  \filldraw[black] (3,2.5) circle (2.5pt) node[anchor=south]{$2$};
   \filldraw[black] (6,0) circle (2.5pt) node[anchor=west]{$3$};
\filldraw[black] (5,-3) circle (2.5pt) node[anchor=north]{$4$};
 \filldraw[black] (1,-3) circle (2.5pt) node[anchor=north]{$5$};
  \filldraw[black] (1.5,0) circle (2.5pt) node[anchor=south]{$a$};
   \filldraw[black] (4.5,-0.5) circle (2.5pt) node[anchor=south]{$b$};
 \filldraw[black] (2,-2) circle (2.5pt) node[anchor=north]{$c$};
  \filldraw[black] (3,1) circle (2.5pt) node[anchor=west]{$d$};
   \filldraw[black] (4,-2) circle (2.5pt) node[anchor=west]{$e$};  
\end{scope}

\begin{scope}
\draw[ thick] (1,2) -- (0,3);   
\draw[ thick] (5,1.8) -- (6,3);
\draw[ thick] (3,-3.5) -- (3,-5);
\draw[ thick] (0.1,-1.75) -- (-1.8,-1.85);
\draw[ thick] (5.85,-1.75) -- (7.5,-2);
\draw[ thick] (0,0) -- (1.5,-0.25);
\draw[ thick] (3,2.5) -- (3,1);
\draw[ thick] (6,0) -- (4.5,-0.5);
\draw[ thick] (5,-3) -- (4,-2);
\draw[ thick] (1,-3) -- (2,-2);
\draw[ thick] (1.5,-0.25) -- (3,1);
\draw[ thick] (3,1) -- (4.5,-0.5);
\draw[ thick] (4,-2) -- (4.5,-0.5);
\draw[ thick] (2,-2) -- (4,-2);
\draw[ thick] (2,-2) -- (1.5,-0.25); 
\draw[ thick] (0,0) -- (1,2);
\draw[ thick] (3,2.5) -- (1,2);
\draw[ thick] (3,2.5) -- (5,1.8);
\draw[ thick] (6,0) -- (5,1.8);
\draw[ thick] (6,0) -- (5.85,-1.75);
\draw[ thick] (5,-3) -- (5.85,-1.75);
\draw[ thick] (3,-3.5) -- (5,-3);
\draw[ thick] (3,-3.5) -- (1,-3);
\draw[ thick] (0,0) -- (0.1,-1.75);
\draw[ thick] (1,-3) -- (0.1,-1.75);
\draw[ thick] (0,3) .. controls (3, 4)  .. (6,3);
\draw[ thick] (7.5,-2) .. controls (7.1, 1)  .. (6,3);
\draw[ thick] (3,-5) .. controls (5.7, -4.1)  .. (7.5,-2);
\draw[ thick] (3,-5) .. controls (0, -4.5)  .. (-1.8,-1.85);
\draw[ thick] (0,3) .. controls (-1.5, 1)  .. (-1.8,-1.85);
 \filldraw[black] (0,0) circle (2.5pt) node[anchor=east]{$1$};
  \filldraw[black] (3,2.5) circle (2.5pt) node[anchor=south]{$3$};
   \filldraw[black] (6,0) circle (2.5pt) node[anchor=west]{$5$};
\filldraw[black] (5,-3) circle (2.5pt) node[anchor=north]{$2$};
 \filldraw[black] (1,-3) circle (2.5pt) node[anchor=north]{$4$};
  \filldraw[black] (1.5,-0.25) circle (2.5pt) node[anchor=south]{$a$};
   \filldraw[black] (4.5,-0.5) circle (2.5pt) node[anchor=south]{$c$};
 \filldraw[black] (2,-2) circle (2.5pt) node[anchor=north]{$e$};
  \filldraw[black] (3,1) circle (2.5pt) node[anchor=west]{$b$};
   \filldraw[black] (4,-2) circle (2.5pt) node[anchor=west]{$d$};
    \filldraw[black] (0,3) circle (2.5pt) node[anchor=south]{$d'$};
     \filldraw[black] (-1.8,-1.85) circle (2.5pt) node[anchor=east]{$c'$};
      \filldraw[black] (7.5,-2) circle (2.5pt) node[anchor=west]{$a'$};
       \filldraw[black] (3,-5) circle (2.5pt) node[anchor=north]{$b'$};
 \filldraw[black] (6,3) circle (2.5pt) node[anchor=west]{$e'$};
  \filldraw[black] (1,2) circle (2.5pt) node[anchor=south]{$2'$};
   \filldraw[black] (5,1.8) circle (2.5pt) node[anchor=west]{$4'$};
    \filldraw[black] (5.85,-1.75) circle (2.5pt) node[anchor=east]{$1'$};
     \filldraw[black] (3,-3.5) circle (2.5pt) node[anchor=south]{$3'$};
     \filldraw[black] (0.1,-1.75) circle (2.5pt) node[anchor=west]{$5'$};
   \end{scope}
    \end{tikzpicture}
    \caption{Dodecahedron (left) is a (double) cover of the Petersen graph (right).}
    \label{fig:petersencover}
\end{figure}
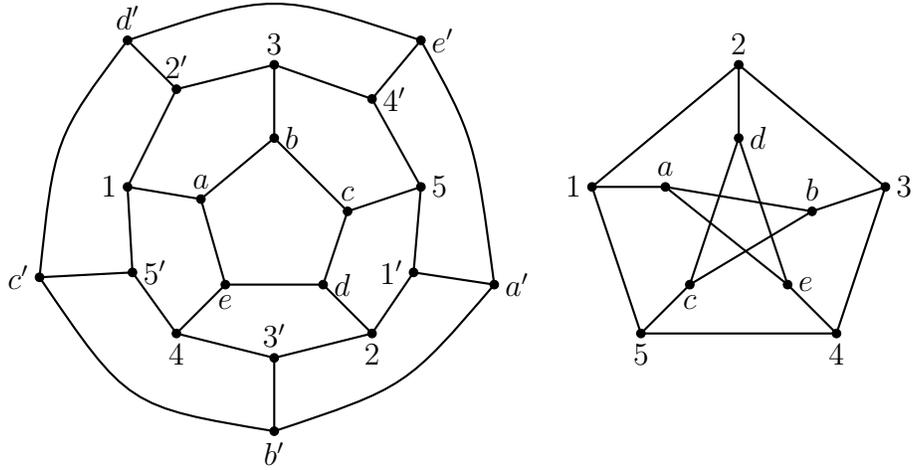

A lot of attention has been paid to graph covers from various perspectives. As purely graph-theoretic objects, graph covers were studied in \cite{gross2001topological,negami1988spherical,archdeacon1990parity}. Graph covers have also appeared in the literature as ``lifts'' of graphs, see \cite{amit2001random}.

Let $S$ be a subset of the vertex set of graph $F$. $S$  is called a \emph{dominating set} of $F$ if
every vertex not in $S$ is adjacent to at least one vertex in $S$. 
The \emph{domination number} of $F$, denoted by $\gamma(F)$, is the minimum cardinality of a dominating set of $F$. A \emph{total dominating set} of $F$ is a set
$S$ such that every vertex is adjacent to some vertex in $S$. Note that by the definition of a total dominating set, it is evident that a graph $F$ admits a total dominating set if and only if $F$ has no isolated vertices.
The \emph{total domination number} of $F$, denoted by $\gamma_t(F)$, is the minimum cardinality of a total dominating set of $F$. A \emph{connected dominating set} $S$ of $F$ is a set such that every vertex not in $S$ is adjacent with at least one vertex in $S$ and the subgraph induced by $S$ is connected. The \emph{connected domination number} $\gamma_c(F)$ is the minimum cardinality of a connected dominating set of $F$.

We saw in  Figure~\ref{fig:petersencover} that the dodecahedron is a cover of the Petersen graph.  One can check that $\gamma(\mathrm{Petersen\; graph}) = 3$ and $\gamma_t(\mathrm{Petersen \; graph}) = \gamma_c(\mathrm{Petersen \; graph}) = 4$. Moreover, $\gamma(\mathrm{Dodecahedron}) = 6, \gamma_t(\mathrm{Dodecahedron}) = 8$ and $\gamma_c(\mathrm{Dodecahedron}) = 10$.

The theory of domination in graphs is one of the main research areas in graph theory. 
For an excellent treatment of fundamentals of domination in graphs and recent topics, we refer the reader to the following books \cite{haynes2023domination,haynes2020topics,haynes2013fundamentals}.

One fundamental area of study in domination in graphs is how domination parameters change under graph operations. For example, Vizing \cite{vizing1968some} studied the connection between the domination number and the cartesian product of graphs, which led to Vizing's conjecture, which many believe to be the main open problem in the area of domination in graphs. Also, see \cite{,sumner1991critical,brigham1988vertex,fisher1998hamiltonicity,gravier1995domination,zwierzchowski2004domination} for other similar research. Graph covers give an interesting way to construct new graphs that are locally isomorphic to a given graph. As a result, graph covers have a lot of intriguing relationships with different types of graph parameters. For example, connections between connectivity and graph covers were carried out in \cite{annor2024three, negami2024another,amit2002random}. Also, see \cite{amit2001random} for other interesting graph-theoretic connections.

The preceding paragraph serves as the main motivation for this work.
Let graph $G$ be a cover of graph $F$. Our goal is to study the relation between $\gamma(G)$ and $\gamma(F)$, $\gamma_t(G)$ and $\gamma_t(F)$, and $\gamma_c(G)$ and $\gamma_c(F)$. To our knowledge, this is the first paper to discuss the connections between domination parameters and graph covers. 

Now we present theorems that are useful in our results. Note that from Proposition~\ref{prop:pertiesgc}\eqref{it:dgpreserve} we have $\Delta(F) = \Delta(G)$ and $\delta(G) = \delta(F)$. So, we write $\Delta$ and $\delta$ to mean the maximum and minimum degree, respectively, without specifying the graph.

\begin{thm}[\cite{walikar1979recent}\label{thm:deltalow}]
If $F$ is a graph of order $n$, then $\gamma(F) \geq \frac{n}{1 + \Delta}$. 
\end{thm}

\begin{thm}[\cite{reed1996paths}]\label{thm:3delta}
 If $F$ is a graph of order $n$ with $\delta \geq 3$, then $\gamma(F) \leq \frac{3}{8}n$.  
\end{thm}

\begin{thm}[\cite{bujtas2019domination}]\label{thm:5delta}
 If $F$ is a graph of order $n$ with $\delta \geq 5$, then $\gamma(F) \leq \frac{1}{3}n$.  
\end{thm}

\begin{thm}[\cite{cockayne1980total, walikar1979recent}]\label{thm:Deltalow}
If $F$ is a connected graph of order $n$, then $\gamma_t(F) \geq \frac{n}{ \Delta}$. 
\end{thm}

\begin{thm}[\cite{archdeacon2004some,chvatal1992small,tuza1990covering}]\label{thm:Deltaup}
If $F$ is a graph of order $n$ with $\delta \geq 3$, then $\gamma_t(F) \leq \frac{1}{2}n$. 
\end{thm}

\begin{thm}[\cite{eustis2016independence}]\label{thm:Deltaup5}
If $F$ is a graph of order $n$ with $\delta \geq 5$, then $\gamma_t(F) \leq \frac{2453}{6500}n$. 
\end{thm}

\subsection*{Greedy Algorithm}
There is an intuitive algorithm which constructs a dominating subset of vertices of a graph $F$ whose cardinality differs from $\gamma(F)$ by no more than, roughly, $\log \Delta$.  

Given a graph $F$. Let $S$ be a subset of the vertex set of $F$. We start with $S = \emptyset$. Until $S$ is a dominating set of $F$, add a vertex $v$ with maximum number of undominated neighbours. All vertices of $F$ start  as \emph{white}. Anytime we add a vertex to $S$, we colour that vertex \emph{black}, and colour all dominated vertices \emph{grey}. Note that all undominated vertices remain white.
Let $w(u)$ be the cardinality of the white vertices among the neighbours of $u$, including $u$ itself. This idea leads to a greedy algorithm.


\begin{algorithm}
\caption{Greedy algorithm}\label{alg:cap}
\begin{enumerate}
    \item $S := \emptyset$
    \item \textbf{while} $\exists$ white vertices \textbf{do}
    \item choose $v \in \{x \;|\; w(x) = \mathrm{max}_{u \in V}\{(w(u)\} \}$; 
    \item  $S:= S \cup \{v\}$;
    \item \textbf{end while}
    \end{enumerate}

\end{algorithm}

\begin{thm}[\cite{vazirani2002approximation}]\label{thm:greedyalg}
The greedy algorithm (see Algorithm~\ref{alg:cap})  is a $H(\Delta)$ factor approximation algorithm for the
minimum dominating set problem, where $H(\Delta) = 1 +\frac{1}{2}+\frac{1}{3}+\cdots +\frac{1}{\Delta}$.
\end{thm}

\section{Results}

This section is devoted to establishing relations concerning the domination number, the total domination number, and the connected domination number with respect to graph covers. We start with the domination number. 

\begin{thm}\label{thm:domn}
Suppose that a graph $G$ is a $k$-fold cover of a graph $F$. Then $\gamma(G) \leq k\gamma(F)$. Moreover, this bound is tight.
\end{thm}

\begin{proof}
Suppose that $\gamma(F) = t$. Let $\pi$ be the projection map, and let $S = \{v_1, v_2,\dots,v_t \}$ be a minimal dominating set of $F$ such that $|S| = t$. We define the set $S' = \{\pi^{-1}(v)\; |\; v \in S \}$. We show that $S'$ is a dominating set of $G$. It is sufficient to show that $v' \in V(G)\setminus S'$ has at least one neighbour in $S'$. Suppose for a contradiction that there is $v' \in V(G) \setminus S'$ such that $v'$ has no neighbour in $S'$. Then $\pi(v')$ has no neighbour in $S$, but $S$ was chosen to be a dominating set with the smallest cardinality; a contradiction. Hence $v'$ has a neighbour in $S'$ and furthermore, $S'$ is a dominating set. It follows that $\gamma(G) \leq |S'| = kt$.

The bound is trivially tight because if $F$ is any graph and $G$ is $k$ disconnected copies of $F$, then $\gamma(G) = k\gamma(F)$.
\end{proof}

A natural question that arises is: can we find a lower bound for $\gamma(G)$ in terms of $k$ and $\gamma(F)$? 

We make the following observations.
\begin{observation}\label{obs:dom}
~
\begin{enumerate}[label=(\roman*),ref=\roman*]
    \item\label{obs:k-bigger} If a graph $G$ is a $k$-fold cover of a graph $F$, then $\gamma(G) \geq k$ by Theorem~\ref{thm:deltalow}.
    \item\label{obs:g-bigger}  If a graph $G$ is a $k$-fold cover of a graph $F$, then $\gamma(G) \geq \gamma(F)$.
\end{enumerate}    
\end{observation}
   
The following is an immediate consequence of Observation~\ref{obs:dom}.

\begin{remark}\label{prop:1stlowbd}
Suppose that a graph $G$ is a $k$-fold cover of a graph $F$. Then $\gamma(G) \geq \sqrt{k\gamma(F)}$. 
\end{remark}


The bound in Remark~\ref{prop:1stlowbd} is not strong. 
In the following, we shall prove a stronger lower bound, which depends on $\Delta$. 

\begin{thm}\label{thm:Deltalowbd}
Suppose that a graph $G$ is a $k$-fold cover of a graph $F$. Then $\gamma(G) \geq \frac{1}{H(\Delta)}k\gamma(F)$,  where $H(\Delta) = 1 +\frac{1}{2}+\frac{1}{3}+\cdots +\frac{1}{\Delta}$. 
\end{thm}

\begin{proof}
 Let $S = \{v_1, v_2, \dots, v_\ell\}$ be a set of vertices in a dominating set of $F$ chosen by the Greedy algorithm, where $v_i$ was chosen before $v_{i+1}$, for $i= 1, 2, \cdots, \ell-1$. Then $\gamma(F)\leq |S|$, so $k\gamma(F)\leq k|S|=|\pi^{-1}(S)|$, (where $\pi$ is the projection map for the cover).

We will show that the lift of $S$ into $G$ could be the set chosen by the greedy algorithm.  Let $S' = \emptyset$ be a subset of the vertex set of $G$. Note that all the vertices of $G$ are white. Since $v_1 \in S$ was chosen first among all the vertices of $F$, 
the greedy algorithm could choose vertices of $G$ that are in $\pi^{-1}(v_1)$ because each vertex in $\pi^{-1}(v_1)$ together with their neighbours form a collection of disjoint stars in $G$. 
Since there are $k$ vertices in $\pi^{-1}(v_1)$, the greedy algorithm could add all the $k$ preimages $\pi^{-1}(v_1)$ of $v_1$ to $S'$ . So, $S':= S' \cup \{v'_{11}, v'_{12}, \dots,v'_{1k} \}$, where $v'_{1i}$, for $i = 1, 2, \dots, k$, are vertices in $\pi^{-1}(v_1)$. Note that $\pi$ maps the neighbours of $v'_{1i}$, (for $i = 1, 2, \dots, k$) bijectively onto the neighbours of $v_1$. So, $w(u)$ maintains the same value as $w(\pi^{-1}(u))$, for all $u \in S\setminus  v_1$.

Now, the next vertex in $S$ is $v_2$. Again, the greedy algorithm could choose vertices of $G$ that are in $\pi^{-1}(v_2)$. Since there are $k$ preimages $\pi^{-1}(v_2)$ of $v_2$, $S'$ becomes $S' = S'  \cup \{v'_{11}, v'_{12}, \dots,v'_{1k} \} \cup \{v'_{21}, v'_{22}, \dots,v'_{2k}\}$,  where $v'_{1i}$ and $v'_{2,i}$ are vertices in $\pi^{-1}(v_1)$ and $\pi^{-1}(v_2)$, respectively, for $i = 1, 2, \dots, k$. By repeating this process for the sequence of vertices in $S$, it follows that the lift of $S$ into $G$ could be a chosen set by the greedy algorithm for $G$. Thus, from Theorem~\ref{thm:greedyalg} we have $|\pi^{-1}(S)| \leq \gamma(G)H(\Delta)$ and so, $k\gamma(F) \leq \gamma(G)H(\Delta)$.  
This proves the theorem.   
\end{proof}

The following result is a consequence of Theorem~\ref{thm:Deltalowbd}.

\begin{cor}\label{conr:cycle}
 Let a graph $G$ be a $k$-fold cover of $C_n$, a cycle on $n$ vertices. Then $\gamma(G) \geq \frac{2}{3}k\gamma(C_n)$.   
\end{cor}

\begin{proof}
    Note that $H(2) = \frac{3}{2}$, and the result follows from Theorem~\ref{thm:Deltalowbd}.
\end{proof}

For $r$-regular graphs with $3 \leq r \leq 5$, we have the following lower bound. 

\begin{thm}\label{thm:3/4domreg}
Suppose that a graph $G$ is a $k$-fold cover of a graph $F$. Then 

\begin{enumerate}[label=\emph{(\alph*)},ref=\alph*]
    \item \label{it:3dom}
    $\gamma(G) \geq \frac{3}{5}k\gamma(F)$ if $F$ is a $3$-regular graph.
\\
    \item \label{it:45dom}
    $\gamma(G) \geq \frac{1}{2}k\gamma(F)$ if $F$ is a $4$ or $5$-regular graph.
\end{enumerate}

\end{thm}

\begin{proof}
For assertion (\ref{it:3dom}), suppose for a contradiction that $\gamma(G) < \frac{3}{5}k\gamma(F)$. 
Let $F$ be a $3$-regular graph on $n$ vertices. Note that $G$ has $kn$ vertices. Then by applying Theorems~\ref{thm:deltalow} and~\ref{thm:3delta} 
   we have
\begin{equation*}
  \frac{1}{4}kn \leq \gamma(G)< \frac{3}{5}k\gamma(F) \leq  \frac{9}{40}kn, 
\end{equation*}
which leads to a contradiction.

For assertion (\ref{it:45dom}), the proof is similar to the proof for assertion (\ref{it:3dom}) by using Theorems~\ref{thm:deltalow}, ~\ref{thm:3delta} and~\ref{thm:5delta}.
\end{proof}

For $r = 3,4$ and $5$, Theorem~\ref{thm:Deltalowbd} gives the following leading coefficients $\frac{6}{11}, \frac{12}{25}$ and $\frac{60}{137}$, respectively. 
So, Theorem~\ref{thm:3/4domreg} gives better leading coefficients than Theorem~\ref{thm:Deltalowbd} when $r =3, 4,5$. Moreover, we do not have examples to illustrate that the bounds in Theorems~\ref{thm:Deltalowbd} and~\ref{thm:3/4domreg} are tight. Therefore, we think that the bound could be improved generally. 

\begin{figure}[h]
   \centering
\begin{tikzpicture}[scale=1.25]

 \draw[ thick ] (0,0)--(4,0);
  \draw[ thick ] (1,1)--(1,-3);
   \draw[ thick ] (2,1)--(2,-3);
 \draw[ thick ] (3,1)--(3,-3);
  \draw[ thick ] (0,-1)--(4,-1);
 \draw[ thick ] (0,-2)--(4,-2);
 
\draw[thick, fill=black] (2,-1) circle (2pt);
\draw[thick, fill=black] (1,0) circle (2pt);
\draw[thick, fill=black] (3,0) circle (2pt);

\filldraw[thick, fill=white] (2,0) circle (2pt);
\filldraw[thick, fill=white] (1,-1) circle (2pt);
\filldraw[thick, fill=white] (3,-1) circle (2pt);
\filldraw[thick, fill=white] (1,-2) circle (2pt);
\filldraw[thick, fill=white] (2,-2) circle (2pt);
\filldraw[thick, fill=white] (3,-2) circle (2pt);

\end{tikzpicture}
   \caption{The graph $N$.}
  \label{fig:domsharp}
\end{figure}

\begin{example}\label{exa:conjecture}
In Figure~\ref{fig:domsharp}, the graph $N$ is the Cartesian product of $C_3$ and $C_3$, where the black vertices are dominating vertices. So $\gamma(N) \leq 3$. Note that $N$ is $4$-regular. So, we can choose any vertex to be in the dominating set and that vertex dominates four other vertices. Now, the remaining four undominated vertices lie on a rectangle in $N$. So we need at least two additional vertices to dominate all of them. Hence, $\gamma(N) = 3$. 
Let a graph $G$ be the Cartesian product of $C_{15}$ and $C_{15}$. Clearly, $G$ is a cover of $N$ with fold $25$. We construct a dominating set $S$ of $G$ as follows. Choose any vertex, say $v$, of $G$ as the starting point. Then $v$ dominates four other vertices. We translate $v$ both vertically and horizontally with a period of $5$ \cite{golomb1970perfect}. We iterate this process for every new vertex we select. Since $G$ is $4$-regular and every vertex lies on $C_{15}$ , we select three vertices on every cycle of length $15$. So, every vertex in $G$ is dominated by exactly one vertex of $S$, and therefore, $S$ is a dominating set. Thus, $\gamma(G) = \frac{15^{2}}{5} = 45$ and moreover, $\gamma(G) = \frac{3}{5}k\gamma(N)$.
\end{example}

We now turn our attention to the total domination number. We prove a result analogous to Theorem~\ref{thm:domn}.

\begin{thm}\label{thm:tdomn}
Let $F$ be a graph without isolated vertices. Suppose that $G$ is a $k$-fold cover of $F$. Then $\gamma_t(G) \leq k\gamma_t(F)$. Also, this bound is tight.
\end{thm}

\begin{proof}
The proof of the inequality is similar to the proof for Theorem~\ref{thm:domn}. The bound is trivially tight if we take $G$ to be a disconnected union of $k$ copies of $F$.
\end{proof}



    

We make the following observations, which will be used to prove a general lower bound for $\gamma_t(G)$.

\begin{observation}\label{obs:tdom}
~
 \begin{enumerate}[label=(\roman*),ref=\roman*]
    \item\label{obs:k-tbigger} If a graph $G$ is a $k$-fold cover of graph $F$ without isolated vertices, then $\gamma_t(G) \geq k$ from Theorem~\ref{thm:Deltalow}.
    \item\label{obs:g-tbigger}  If a graph $G$ is a $k$-fold cover of graph $F$  without isolated vertices, then $\gamma_t(G) \geq \gamma_t(F)$. 
\end{enumerate}   
\end{observation}


\begin{remark}\label{prop:1sttlowbd}
Let $F$ be a graph without isolated vertices. Suppose that a graph $G$ is a $k$-fold cover of $F$. Then $\gamma_t(G) \geq \sqrt{k\gamma_t(F)}$. 
\end{remark}


For $r$-regular graphs with $3 \leq r \leq 5$, we have the following lower bound.



\begin{thm}\label{thm:3/4tdomreg}
Suppose that a graph $G$ is a $k$-fold cover of a graph $F$. Then 

\begin{enumerate}[label=\emph{(\alph*)},ref=\alph*]
    \item \label{it:3tdom}
    $\gamma_t(G) \geq \frac{3}{5}k\gamma_t(F)$ if $F$ is a $3$-regular graph.
\\
    \item \label{it:45tdom}
    $\gamma_t(G) \geq \frac{1}{2}k\gamma_t(F)$ if $F$ is a $4$ or $5$-regular graph.
\end{enumerate}
\end{thm}

\begin{proof}
Using Theorems~\ref{thm:Deltalow}, \ref{thm:Deltaup}, and~\ref{thm:Deltaup5}, one can prove the result in the same way as the proof for Theorem~\ref{thm:3/4domreg}.
\end{proof}

Finally, we consider the connected domination number. Note that by definition, a graph admits a connected dominating set if and only if the graph is connected.

We observe the following.

\begin{observation}\label{obs:cdom}
~
\begin{enumerate}[label=(\roman*),ref=\roman*]
    \item\label{obs:ck-bigger} If a connected graph $G$ is a $k$-fold cover of a connected graph $F$, then $\gamma_c(G) \geq k$ because $\gamma_c(G) \geq \gamma(G)$.
    \item\label{obs:cg-bigger}  If a connected graph $G$ is a $k$-fold cover of a connected graph $F$, then $\gamma_c(G) \geq \gamma_c(F)$.
\end{enumerate}
\end{observation}

\begin{remark}\label{prop:1stclowbd}
Let $F$ and $G$ be connected graphs. If $G$ is a $k$-fold cover of $F$, then $\gamma_c(G) \geq \sqrt{k\gamma_c(F)}$. 
\end{remark}


\begin{thm}\label{thm:cloupdom}
Let $F$ and $G$ be connected graphs. Suppose that $G$ is a $k$-fold cover of $F$. Then 
$\gamma_c(G) \leq k(\gamma_c(F) + 2) -2$. Moreover, this bound is tight.
\end{thm}

\begin{proof}
Let $S$ be a connected dominating set of $F$ with a minimum cardinality. Then, by the definition, the elements of $S$ induce a connected subgraph $F'$ of $F$. Let $T'$ be a spanning tree of $F'$. Then from Proposition~\ref{prop:pertiesgc}~\eqref{it:compt}, $\pi^{-1}(T')$ is a disjoint union of trees of $G$ isomorphic to $T'$. Let $T_1, … ,T_k$ be the connected components of $\pi^{-1}( T’)$. Note that the set of vertices of the subgraph $\cup_{i=1}^k T_i \subset G$ is a dominating set of $G$. We construct a connected subgraph $G'$ of $G$ containing the vertices of $\cup_{i=1}^k T_i$ and having no more than $k(\gamma_c (K) + 2) - 2$ vertices (note that $\cup_{i=1}^k T_i$ has $k \gamma_c (F)$ vertices). The subgraph $G'$ is constructed by induction. We start with setting $G'_1 = T_1$ and with considering a shortest path $p_1$ between $G'_1$ and $\cup_{i=2}^k T_i$. Up to relabelling, we can assume that the endpoint of $p_1$ not belonging to $G'_1$ belongs to $T_2$. The path $p_1$ cannot have more than three edges, as otherwise it would contain a vertex not adjacent to any vertices of $\cup_{i=1}^k T_i$. We define $G'_2 = G'_1 \cup p_1 \cup T_2$, and consider a shortest path $p_2$ between $H_2$ and $\cup_{i=3}^k T_i$. The path $p_2$ cannot have more than $3$ edges. Up to relabelling we can assume that the endpoint of $p_2$ not belonging to $G'_2$ belongs to $T_3$, and we define $G'_3 = G'_2 \cup p_2 \cup T_3$. Repeating this procedure, we obtain a connected subgraph $G'=G'_{k-1} \subset G$ which contains $\cup_{i=1}^k T_i$ and whose number of vertices is at most $k \gamma_c(F) + 2(k-1)$, as required.

If $F = C_n$ (for $n \geq 3$), then $G = C_{kn}$. It is known that $\gamma_c(C_n) = n-2$. So, $\gamma_c(G) = kn-2 = k(n -2 + 2)-2$, which proves that the bound is tight.
\end{proof}

\begin{figure}[h]
    \centering
    \begin{tikzpicture}[scale=0.9]
  \draw[ thick ] (-8,0)--(0,0); 
   \draw[ thick ] (-7,-1)--(1,-1); 
    \draw[ thick ] (-7,-2)--(1,-2); 
     \draw[ thick ] (-8,-3)--(0,-3);
 \draw[ thick ] (-8,0)--(-8,-3); 
  \draw[ thick ] (-8,0)--(-7,-1); 
   \draw[ thick ] (-8,-3)--(-7,-2); 
 \draw[ thick ] (-7,-1)--(-7,-2); %
      \draw[ thick ] (-6,0)--(-6,-3);
 \draw[ thick ] (-6,0)--(-5,-1); 
  \draw[ thick ] (-6,-3)--(-5,-2); 
   \draw[ thick ] (-5,-1)--(-5,-2); %
      \draw[ thick ] (-4,0)--(-4,-3);
 \draw[ thick ] (-4,0)--(-3,-1); 
  \draw[ thick ] (-4,-3)--(-3,-2); 
   \draw[ thick ] (-3,-1)--(-3,-2);%
      \draw[ thick ] (-2,0)--(-2,-3);
 \draw[ thick ] (-2,0)--(-1,-1); 
  \draw[ thick ] (-2,-3)--(-1,-2); 
   \draw[ thick ] (-1,-1)--(-1,-2);%
      \draw[ thick ] (0,0)--(0,-3);
 \draw[ thick ] (0,0)--(1,-1); 
  \draw[ thick ] (0,-3)--(1,-2); 
   \draw[ thick ] (1,-1)--(1,-2);
\filldraw[black] (-8,-3) circle (3pt);
\draw[thick, fill=black] (-7,-2) circle (3pt);
\draw[thick, fill=black] (-5,-2) circle (3pt);
\draw[thick, fill=black] (-3,-2) circle (3pt);
\draw[thick, fill=black] (-1,-2) circle (3pt);
\draw[thick, fill=black] (1,-2) circle (3pt);
\draw[thick, fill=black] (-6,0) circle (3pt);
\draw[thick, fill=black] (-2,0) circle (3pt);
\draw[thick, fill=black] (-5,-1) circle (3pt);
\draw[thick, fill=black] (-1,-1) circle (3pt);
\filldraw[thick, fill=white] (-8,0) circle (3pt);
\filldraw[thick, fill=white] (-4,0) circle (3pt);
\filldraw[thick, fill=white] (0,0) circle (3pt);
\filldraw[thick, fill=white] (-6,-3) circle (3pt);
\filldraw[thick, fill=white] (-2,-3) circle (3pt);
\filldraw[thick, fill=white] (-4,-3) circle (3pt);
\filldraw[thick, fill=white] (0,-3) circle (3pt);
\filldraw[thick, fill=white] (-7,-1) circle (3pt);
\draw[thick, fill=black] (-5,-1) circle (3pt);
\filldraw[thick, fill=white] (-3,-1) circle (3pt);
\filldraw[thick, fill=white] (1,-1) circle (3pt);
    \end{tikzpicture}
    \caption{The graph $H$.}
    \label{fig:conloweam}
\end{figure}

\begin{example} 
In Figure~\ref{fig:conloweam}, the black vertices are connected dominating set of a graph $H$. So, $\gamma_c(H) \leq 9$. Without loss of generality, we consider the top horizontal line in Figure~\ref{fig:conloweam}. We can dominate the five vertices with exactly two vertices, and the two vertices cannot be adjacent. Since we are forming a connected dominating set, each of the two vertices must be adjacent to a vertex in the connected dominating set. So we select additional two vertices on the second-top horizontal line. Again, we can dominate the third-top horizontal line with two vertices and they are not adjacent. So, we have to select their common adjacent vertex to be in the connected dominating set. Finally, we dominate the bottom horizontal line with two vertices.  Thus, $\gamma_c(H) = 9$. Let a graph $G$ be 
the $5\times8$ grid with an edge connecting every pair $u_i, v_i$, for $i = 1, 2, \dots, 5$, where $u_i, v_i$ (for $i = 1, 2, \dots, 5$) are the vertices of the opposite ends of the $5 \times 8$ grid, and $u_i, v_i$ lie on the same straight line in the grid. 
Then $G$ is a $2$-fold cover of graph $H$. The connected domination number of the $5 \times 8$ grid is $17$ (see \cite{goto2021connected}) and hence $\gamma_c(G) \leq 17$. But $17 < 2 \times 9 = 18$. So in general, it is not true that $\gamma_c(G) \geq k\gamma_c(F)$ if $G$ is a $k$-fold cover of  $F$, where $G$ and $F$ are connected graphs.
\end{example}

\section{Conclusion}
In this paper, we investigated how three domination parameters behave under graph covers. 
We hope that another major direction of research on domination in graphs is suggested by the results. 
Furthermore, our results suggest that the problem of characterising graphs that achieve equality in any of the upper bounds would be interesting. 








\section*{Acknowledgements}
The author extends his gratitude to Michael Payne and Yuri Nikolayevsky for their invaluable discussions, suggestions, and continuous encouragement throughout this paper. Moreover, the author would like to express his thanks to Shiksha Shiksha for proofreading a draft of the paper. 






 \bibliographystyle{plain}
\bibliography{references}

\end{document}